\makeatletter
	\def\@fnsymbol#1{\ensuremath{\ifcase#1\or \dagger\or \ddagger\or\mathsection\or \mathparagraph\or \|\or **\or \dagger\dagger\or \ddagger\ddagger \else\@ctrerr\fi}}
\makeatother

\title{Strong solutions to McKean--Vlasov SDEs with coefficients of Nemytskii-type}
\author{Sebastian Grube\thanks{
                  Faculty of Mathematics, Bielefeld University, 33615 Bielefeld, Germany. E-Mail: sgrube@math.uni-bielefeld.de}
        }
\documentclass{article}
\usepackage[a4paper, left=3.55cm, right=3.55cm, top=2cm]{geometry}
\usepackage{setspace}
\usepackage{graphicx}
\usepackage[utf8]{inputenc}
\usepackage{amsmath}
\usepackage{amsfonts}
\usepackage{amssymb}
\usepackage{amsthm}
\usepackage{hyperref}
\usepackage{mathrsfs}
\usepackage{bbm}
\usepackage{esint}
\usepackage{enumerate}
\usepackage{soul}

\newcommand{\law}[1]{{\mathcal L}_{#1}}
\newcommand{\LN}{\mathbb N}
\newcommand{\RR}{\mathbb R}

\newcommand{\BBB}{\mathbb B}
\newcommand{\BBBB}{\mathcal B}
\newcommand{\EEEE}{\mathcal E}
\newcommand{\FF}{\mathscr F}
\newcommand{\NN}{\mathcal{N}}
\newcommand{\PP}{\mathbb P}
\newcommand{\QQ}{ Q}
\newcommand{\PPPP}{\mathcal P}
\newcommand{\SSS}{\mathscr S}
\newcommand{\WW}{\mathbb{W}}

\newcommand{\inv}{^{-1}}
\newcommand{\rd}{{\RR^d}}
\newcommand{\scalarproduct}[3][]{\langle #2, #3\rangle_{#1}}
\newcommand{\divv}{\mathrm{div}}

\usepackage{nicefrac}

\usepackage{xifthen}

\newcommand{\norm}[2][]{\ifthenelse{\isempty{#1}}
	{\left\lVert#2\right\rVert}
	{\left\lVert#2\right\rVert_{{#1}}}
}

\newtheorem{theorem}{Theorem}[section]

\newtheorem{definition}[theorem]{Definition}
\newtheorem{remark}[theorem]{Remark}

\begin{document}
\newpage
\maketitle
\begin{abstract}
We study a large class of McKean--Vlasov SDEs with drift and diffusion coefficient depending on the density of the solution's time marginal laws in a Nemytskii-type of way.
A McKean--Vlasov SDE of this kind arises from the study of the associated nonlinear FPKE, for which is known that there exists a bounded Sobolev-regular Schwartz-distributional solution $u$.
Via the superposition principle, it is already known that there exists a weak solution to the McKean--Vlasov SDE  with time marginal densities $u$.
We show that there exists a strong solution the McKean--Vlasov SDE, which is unique among weak solutions with time marginal densities $u$.
The main tool is a restricted Yamada--Watanabe theorem for SDEs, which is obtained by an observation in the proof of the classic Yamada--Watanabe theorem.
\end{abstract}

\textbf{Mathematics Subject Classification (2020):} 60H10,
60G17,
35C99.
\\
\textbf{Keywords:} McKean--Vlasov stochastic differential equation, pathwise uniqueness, Yamada--Watanabe theorem, nonlinear Fokker--Planck--Kolmogorov equation

\section{Introduction}
\sloppy In this paper we will consider the following McKean--Vlasov stochastic differential equation (abbreviated by McKean--Vlasov SDE or MVSDE) in $\rd$, $d\in\LN$,  with coefficients of Nemytskii-type, which in our case is of the form
\begin{align}\label{MVSDE}
	dX(t) \notag
	=&\ E(X(t))b\left(\frac{d\law{X(t)}}{dx}(X(t))\right)dt  + \sqrt{2a\left(\frac{d\law{X(t)}}{dx}(X(t))\right)}\mathbbm{1}_{d\times d}\ dW(t), \notag \\
	X(0)=&\ \xi,\tag{MVSDE.PME}
\end{align}
where $t\in [0,T]$, $T\in (0,\infty)$, $\mathbbm{1}_{d\times d}$ is the $d$-dimensional unit matrix, $(W_t)_{t\in[0,T]}$ is a standard $d$-dimensional $(\FF_t)$-Brownian motion and $\xi$ an $\FF_0$-measurable function on some stochastic basis  $(\Omega,\FF,\PP;(\FF_t)_{t\in[0,T]})$, i.e. a complete, filtered probability space, where $(\FF_t)_{t\in[0,T]}$ is a normal filtration, and $\law{X(t)}:= \PP\circ (X(t))\inv$, $t\in [0,T]$.
Here, we assume that
\begin{align*}
	E: \rd \to \rd ,\
	b : \RR  \to \RR,\
	a: \RR \to \RR
\end{align*}
are functions with  $a(r):= \nicefrac{\beta(r)}{r}, r\in\RR\backslash\{0\}$, $a(0):= \beta'(0)$, such that the following assumptions hold:
\begin{enumerate}[(i)]
	\item \label{condition.beta.general} $\beta\in C^1(\RR)$, $\beta(0)=0$;
	\item \label{condition.beta.monotone} There exists $\gamma_0>0$ such that for all $r_1,r_2\in\RR$
	\begin{align*}
		\gamma_0 |r_1-r_2|^2 \leq (\beta(r_1)-\beta(r_2))(r_1-r_2);
	\end{align*}
	\item \label{condition.D} $E \in L^\infty(\rd;\rd)$, $\divv E \in L^2(\rd)+L^\infty(\rd)$, $(\divv E)^- \in L^\infty(\rd)$;
	\item \label{condition.b} $b\in C^1(\RR)\cap C_b(\RR)$, $b\geq 0$.\end{enumerate}
Here, we would like to point out that conditions \eqref{condition.beta.general} and \eqref{condition.beta.monotone} imply
\begin{align}\label{condition.a.nondegenerate} 
	a\geq\gamma_0> 0,
\end{align}
which, in turn, means that the diffusion matrix of \eqref{MVSDE} is assumed to be non-degenerate. 

There is a vast literature on the solvability of McKean--Vlasov SDEs under various assumptions on the coefficients. In 1966, McKean \cite{mckean1966class} initiated the study of diffusion processes related to certain non-linear PDEs arising from, for example, statistical mechanics (as in our case, see \cite{barbu2018Prob}). This work was closely followed up by important results such as \cite{funaki1984}, \cite{sznitman1984}, \cite{scheutzow1987} investigating the weak and/or strong solvability of McKean--Vlasov SDEs; for more references see \cite{barbu2019nonlinear}. For recent results consult, in particular, \cite{mishura2020existence}, \cite{zhang2021ddsde}, \cite{zhang2021ddsde},
	\cite{huang2019singular}, \cite{huang2020mckeanvlasov}, \cite{huang2020wellposedness}, and also \cite{delarue2018mckean} and the references therein.
	In all of these papers, the authors assume the continuity of the coefficients in the measure-component with respect to the weak topology, some Wasserstein distance, or total variation norm. In \cite{huang2020wellposedness}, the authors consider also the coefficients' continuity with respect to a norm, which is stronger than the sum of a Wasserstein distance and total variation norm.
	However, the coefficients in \eqref{MVSDE} do not bear any such continuity property in their measure-component. 

Equation \eqref{MVSDE} arises from the study of a nonlinear Fokker--Planck--Kolmogorov equation (in short: FPKE), which in this case is a porous medium equation perturbed by a nonlinear transport term of the following type
\begin{align}\label{PME}\tag{PME}
	\partial_t u + \divv(Eb(u)u)-\Delta\beta(u)=0\ \ \text{on } [0,T]\times\rd\ \text{ with }
	\left.u\right|_{t=0} = u_0.
\end{align}
This equation is to be understood in the Schwartz-distributional sense. We will say that a curve of
$L^1(\rd)$-functions $u=(u_t)_{t\in [0,T]}$
is a Schwartz-distributional solution to \eqref{PME} if $[0,T]\mapsto u_t(x)dx$ is narrowly continuous and
\begin{align}\label{FPKE.test}
\int_{\RR^d} \varphi(x) u_t(x)dx =& \int_{\RR^d} \varphi(x) u_0(x)dx + \int_0^t \int_{\RR^d} E^i(x)b(u_s(x)) \partial_i \varphi(x) u_s(x)dxds \nonumber\\
 &+ \int_0^t\int_\rd\beta(u_s(x))\Delta \varphi(x) dxds\ \ \forall t\in [0,T],
\end{align}
for each 
$\varphi \in C^{\infty}_c(\RR^d)$ (using Einstein summation convention), where $E(x)=(E^i(x))_{i=1}^d$. In the case that $u_t$ is even a probability density for all $t\in [0,T]$, then $u=(u_t)_{t\in [0,T]}$ is simply called a \textit{probability solution} to \eqref{PME}.

In \cite{barbu2018Prob} and \cite{barbu2019nonlinear}, an approach was developed in order to solve general McKean--Vlasov SDEs by first solving the associated nonlinear FPKE. This approach is based on the superposition principle developed by Trevisan \cite{trevisan_super}, which in turn relies on the fundamental result of Figalli \cite{figalli2008}; for a very recent generalisation of the latter two publications see \cite{bogachev2021super}. Note also the very recent superposition principle, which relates solutions to non-local FPKEs with solutions to SDEs with jumps \cite{roeckner2020levy}.
Of course, solving first the McKean--Vlasov SDE, It\^o's formula yields that the time marginal laws of the solution process solve the associated nonlinear FPKE. In this sense, solving the McKean--Vlasov equation is essentially equivalent to solving the associated nonlinear FPKE.

In the special case when the nonlinear FPKE is of the type \eqref{PME}, in \cite{barbu2020solutions}, Barbu and Röckner solved this equation under various assumptions on $\beta,E,b$ 
 and lifted the solution to a weak solution to \eqref{MVSDE}, whose time marginal laws coincide with the constructed probability solution $u$ to \eqref{PME}, provided $u_0$ is a bounded probability density.
The aim of this paper is to show that the constructed weak solution provided by \cite{barbu2020solutions} is a functional of the driving Brownian motion, i.e. a strong solution to \eqref{MVSDE} under our assumptions. 
Our method relies on a proper modification of the Yamada--Watanabe theorem for SDEs.
This modification makes it possible to prove the existence of a strong solution to an SDE provided weak existence and pathwise uniqueness holds in a certain subclass of weak solutions.
We will give details about this in Section \ref{section.SDE.yamada}, as there seems to be uncertainty about this result; in \cite{champagnat2018}, the authors stated, 'Note also that pathwise uniqueness is proved only for particular solutions [...], so we cannot use directly the result of Yamada and Watanabe to deduce strong existence [for the SDE under investigation].' (\cite[p. 1502]{champagnat2018}).
This modification of the Yamada--Watanabe theorem can be applied to \eqref{MVSDE} by fixing the solution $u$ to \eqref{PME} provided by \cite{barbu2020solutions} in the coefficients of \eqref{MVSDE}. This transfers the problem of strong existence for McKean--Vlasov SDEs to a problem for SDEs.

Moreover, \eqref{MVSDE} has already been studied in terms of weak existence and restricted pathwise uniqueness by Jabir and Bossy in \cite{bossy2019moderated} in the case $E\equiv b \equiv 0$ and under assumptions which strictly imply ours. They did not prove the existence of a strong solution in their case.\\
	
	This paper is structured as follows.\\
	First, we will introduce some frequently used notation in this paper.
	Afterward, in Section \ref{section.SDE.yamada}, we will present an abstract modification of the famous Yamada--Watanabe theorem for SDEs based on \cite[Appendix E]{spde}, which, in particular, enables us to conclude strong existence provided one has proved weak existence and pathwise uniqueness for some subclass of weak solutions. This theorem will be the main tool to deduce the existence of a strong solution to \eqref{MVSDE}.
	In Section \ref{section.MVSDE.PME.yamada}, we will apply the Yamada--Watanabe theorem for SDEs to general McKean--Vlasov SDEs by fixing the time marginal laws of a given curve of probability measures in the coefficients' measure component.
	Within the last section, Section \ref{section.MVSDE.PME.strong}, we will state the main result and its proof. This section is divided into three subsection.
	In Subsection \ref{section.MVSDE.PME.Procedure}, we will state the main result and the steps on how to prove it.
	In Subsection \ref{section.PME}, we will discuss the existence and regularity of a probability solution $u$ to \eqref{PME} under the conditions \eqref{condition.beta.general}-\eqref{condition.b}.
	In Subsection \ref{section.MVSDE.PME.weakExistence}, we will conclude the existence of a weak solution to \eqref{MVSDE} with time marginal law densities $u$.
	In Subsection \ref{section.MVSDE.PME.pathwiseUniqueness} we show that pathwise uniqueness holds among weak solutions to \eqref{MVSDE} with time marginal law densities $u$.
	This subsection is divided into two further subsections.
	In Subsection \ref{section.SDE.pathwiseUniqueness}, we will recall a pathwise uniqueness result for SDEs with bounded Sobolev-regular coefficients.
	In Subsection \ref{section.MVSDE.PME.pathwiseUniqueness} we will apply the pathwise uniqueness result for SDEs from Subsection \ref{section.SDE.pathwiseUniqueness} to \eqref{MVSDE}. Here we will add condition \eqref{condition.a} to the previous assumptions \eqref{condition.beta.general}-\eqref{condition.b}.
\section*{Notation}
Within this paper we will use the following notation.

For a topological space $(\textbf{T},\tau)$, $\BBBB(\textbf{T})$ shall denote the Borel $\sigma$-algebra on $(\textbf{T},\tau)$.\\
Let $n\geq 1$. On $\RR^n$, we will always consider the usual $n$-dimensional Lebesgue measure $\lambda^n$ if not said any differently. If there is no risk for confusion, we will just say that some property for elements in $\RR^n$ holds \textit{almost everywhere} (or $\textit{a.e.}$) if and only if it holds $\lambda^n$-almost everywhere.
Furthermore, on $\RR^n$, $|\cdot|_{\RR^n}$ denotes the usual Hilbert--Schmidt norm. If there is no risk for confusion, we will just write $|\cdot|=|\cdot|_{\RR^n}$. By $B_R(x)$ we will denote the usual open ball with center $x\in \RR^n$ and radius $R>0$. The rational numbers will be denoted by $\mathbb{Q}$.

Let $(M,d)$ be a metric space. Then $\PPPP(M)$ denotes the set of all Borel probability measures on $(M,d)$. We will consider $\PPPP(M)$ as a topological space with respect to the topology of weak convergence of probability measures. 
A curve of probability measures $(\nu_t)_{t\in [0,T]}\subset \PPPP(M)$ is called narrowly continuous if $[0,T]\ni t\mapsto \int \varphi(x)\nu_t(dx)$ is continuous for all $\varphi \in C_b(M)$. By $\PPPP_0(\RR^n)$ we will denote the set of all probability densities with respect to Lebesgue measure, i.e.
\begin{align*}
	\PPPP_0(\RR^n) = \left\{ \rho \in L^1(\RR^n)\ :\ \rho\geq 0 \text{ a.e.}, \int_{\RR^n} \rho(x) dx =1\right\}.
\end{align*}

Let $(\Omega,\FF,\PP)$ be a probability space and $(S,\SSS)$ a measurable space. If $X:\Omega \to S$ is an $\FF\slash\SSS$-measurable function, then we say that $\law{X}:=\PP\circ X\inv$ is the \textit{law} of $X$.

By $C_c^\infty(\RR^n)$ we denote the set of all infinitely differentiable functions with compact support. 
Let $(S,\SSS,\eta)$ be a measure space and $E$ be a Banach space.
The space of continuous functions on the interval $[0,T]$ with values in $E$ are denoted by $C([0,T];E)$.
For $t\in [0,T]$, $\pi_t: C([0,T];E) \to E$ denotes the canonical evaluation map at time $t$, i.e. $\pi_t(w):=w(t), w\in C([0,T];E)$.
Furthermore, for $1\leq q \leq \infty$, $L^q(S;E)$ symbolises the usual Bochner space on $S$ with values in $E$. 
If $S=\RR^n$ and $E=\RR$, we just write $L^p(\RR^n;\RR)=L^p(\RR^n)$.
Moreover, $W^{1,p}(\RR^n)$ denotes the usual Sobolev-space, containing all $L^p(\RR^n)$-functions, whose first-order distributional derivatives can be represented by elements in $L^p(\RR^n)$.
Moreover, $\divv,\Delta,\nabla$ symbolise the divergence, Laplacian and gradient with respect to the spatial variable and are taken in the Schwartz-distributional sense. Further, the transpose of the distributional Jacobian matrix is also denoted by the gradient symbol.
 \section{A modification of the Yamada--Watanabe theorem for SDEs}\label{section.SDE.yamada}
Since it will be the core of the technique of this paper, we will start by presenting a restricted version of the famous Yamada--Watanabe theorem for SDEs.

The well-known Yamada--Watanabe theorem for SDEs (see, e.g. \cite[Appendix E]{spde}; for the orginal work see \cite{yamada1971yamada}) provides a useful characterisation for the existence of a unique strong solution to a stochastic differential equation; therefore, loosely speaking, it is necessary and sufficient to have a weak solution for \textit{any} initial probability measure 
in combination with the pathwise uniqueness regarding \textit{all} weak solutions.

Carefully checking the statements and the proofs in \cite[Appendix E]{spde}, it is possible to refine their definitions and results to a \textit{restricted} Yamada--Watanabe theorem. This theorem is the result of a generalisation of the observation in \cite[Remark E.0.16]{spde}, which implies that, by the techniques employed in \cite[Appendix E]{spde}, a strong solution can be constructed from a weak solution with a fixed initial condition in the case that pathwise uniqueness is known for solutions with exactly this initial datum.

The restricted Yamada--Watanabe theorem will be of the following form.
Let us fix a set $P$ consisting of probability measures on the solution's path space, which have all the same initial time marginal laws.
Assume that pathwise uniqueness holds among all weak solutions to an SDE whose laws lie in $P$ and that there exists a weak solution $(X,W)$ with $\law{X} \in P$. Then, and only then, this is the case if there exists a unique strong solution to this SDE with law in $P$. For the precise statement see Theorem \ref{yamada.strongSol}.

This section is a modification of \cite[Appendix E]{spde} in which we slightly change the definitions, remarks, lemmata, theorems, and proofs to our setting. For reader’s convenience, we will stick to the finite dimensional setting. The adaption for the infinite dimensional case is essentially the same. The proof will be sketchy in the unclear and undetailed in the clear parts of the main result, since most of the technical parts remain unchanged compared to \cite[Appendix E]{spde}.
For the details consult \cite{grube2022thesis}, where also the infinite dimensional case is treated.

For the finite dimensional case of \cite[Appendix E]{spde}, we refer to \cite[Appendix E]{prevot2007spde}.
For the Yamada--Watanabe theorem in the mild solution framework, we refer to \cite{ondrejat2004uniqueness}. For the treatment of general stochastic models see \cite{kurtz2007yamada, kurtz2014yamada}. \\

Let $(\Omega, \FF,\PP;(\FF_t)_{t\geq 0})$ be a stochastic basis and $d_1 \in \LN$.
We will consider the Polish path spaces $(\BBB,\rho)$, $(\WW_0, \rho)$, where
\begin{align*}
	\BBB:= C([0,\infty);\rd), \ \ \WW_0 := \{w\in C([0,\infty);\RR^{d_1}): w(0)=0\},
\end{align*}
are respectively equipped with the metric
\begin{align*}
	\rho(w_1,w_2):=\sum_{k=1}^\infty 2^{-k} \left(\max_{0\leq t\leq k	}	|w_1(t)-w_2(t)|\wedge 1\right).
\end{align*}
The Borel $\sigma$-algebra of $\BBB$ and $\WW_0$ are denoted by $\BBBB(\BBB)$ and $\BBBB(\WW_0)$, respectively. Furthermore, for $t\in [0,\infty)$, we define $\BBBB_t(\BBB):=\sigma(\pi_s:0\leq s\leq t)$, where $\pi_s(w):= w(s), w\in \BBB$. $\BBBB_t(\WW_0)$ is defined analogously.\\

The equation under investigation is the following path-dependent stochastic differential equation
\begin{align}\label{SEE}\tag{SDE.pd}
	dX(t)=b(t,X)dt+\sigma(t,X)dW(t),\ \ t\in [0,\infty),
\end{align}
where  $b: [0,\infty) \times \BBB \to \rd$ and $\sigma: [0,\infty) \times \BBB \to \RR^{d\times d_1}$ are $\BBBB([0,\infty))\otimes \BBBB(\BBB) \slash \BBBB(\rd)$ and $\BBBB([0,\infty))\otimes \BBBB(\BBB) \slash \BBBB(\RR^{d\times d_1})$-measurable, respectively, such that for each $t \in [0,\infty)$
$b(t,\cdot) \text{ is } \BBBB_t(\BBB)\slash \BBBB(\rd)\text{-measurable}$,
and
	$\sigma(t,\cdot) \text{ is } \BBBB_t(\BBB)\slash \BBBB(\RR^{d\times d_1} )\text{-measurable.}$
Furthermore, $W$ is a standard $d_1$-dimensional $(\FF_t)$-Brownian motion and $\PP_W$ denotes the distribution of $W$ on $(\WW_0,\BBBB(\WW_0))$.\\

Let $\mu_0 \in \PPPP(\rd)$ and $P \subset \PPPP(\BBB)$. We will write $P=P_{\mu_0}$, if all the time marginal laws of the measures in $P$ start with the same measure $\mu_0$ at time $t=0$.
\begin{definition}[$P_{\mu_0}$-weak solution]\label{E.0.1}
	A pair $(X,W)$ is called a $P_{\mu_0}$-weak solution to \eqref{SEE}, if $X=(X(t))_{t\geq 0}$ is an $(\FF_t)$-adapted process with paths in $\BBB$, and $W$ is a standard $d_1$-dimensional $(\FF_t)$-Brownian motion on some stochastic basis $(\Omega,\FF,\PP;(\FF_t)_{t\geq 0})$ such that the following holds:
	\begin{enumerate}[(i)]
		\item $\PP\left(\int_0^T |b(s,X)| + |\sigma(s,X)|^2 ds<\infty\right)=1$, for every $T\geq 0$,
		\item The following equation holds
	    \begin{align}\label{SEE.integral}
			X(t) = X(0)+\int_0^t b(s,X) ds + \int_0^t \sigma(s,X)dW(s), \text{ for every } t \geq 0\ \PP\text{-a.s.},
		\end{align}
		\item $\PP\circ X\inv \in P_{\mu_0}$ (in particular, $\PP\circ X(0)\inv=\mu_0$).
	\end{enumerate}
\end{definition}
\begin{remark}\label{E.0.2} From the measurability assumptions on $b$ and $\sigma$, it follows that, if $X$ is as in Definition \ref{E.0.1}, then both processes $b(\cdot,X)$ and $\sigma(\cdot,X)$ are $(\FF_t)$-adapted.
\end{remark}
\begin{definition}[$P_{\mu_0}$-weak uniqueness]
We say that $P_{\mu_0}$-weak uniqueness holds for \eqref{SEE}, if any two $P_{\mu_0}$-weak solutions $(X,W)$, $(X',W')$ on stochastic bases $(\Omega,\FF,\PP;(\FF_t)_{t\geq 0})$ and $(\Omega',\FF',\PP';(\FF'_t)_{t\geq 0})$, respectively, have the same law on $\BBB$, i.e.
	$\PP\circ X\inv = \PP'\circ (X')\inv$.
\end{definition}
\begin{remark}\label{yamada.remark.weakUniq.element}
	Note that if $P_{\mu_0}$-weak uniqueness holds, it does not necessarily mean that $P_{\mu_0}$ consists of only one element, since not any element in $P_{\mu_0}$ needs to be the law of a $P_{\mu_0}$-weak solution. If, however,
	$$P_{\mu_0}=P_{\mu_0}\cap \{Q \in \PPPP(\BBB) : \exists (X,W) \text{\ weak solution to \eqref{SEE}}, \text{ such that } Q = \law{X}\},$$
then $P_{\mu_0}$ contains exactly one element.
\end{remark}
\begin{definition}[$P_{\mu_0}$-pathwise uniqueness]
We say that $P_{\mu_0}$-pathwise uniqueness holds for \eqref{SEE}, if for any two $P_{\mu_0}$-weak solutions $(X,W)$, $(Y,W)$ on a common stochastic basis $(\Omega,\FF,\PP;(\FF_t)_{t\geq 0})$ with a common standard $d_1$-dimensional $(\FF_t)$-Brownian motion $W$, 
$$\text{$X(0)=Y(0)$ $\PP$-a.s. implies $X(t)=Y(t)$ for all $t\geq 0\ \PP$-a.s.}$$
\end{definition}
Let $\tilde{\EEEE}_{\mu_0}$ to be the set of all maps $F_{\mu_0} : \rd\times \WW_0 \to \BBB$ which are $\overline{\BBBB(\rd)\otimes \BBBB(\WW_0)}^{{\mu_0}\otimes \PP_W}\slash\BBBB(\BBB)$-measurable, where $\overline{\BBBB(\rd)\otimes \BBBB(\WW_0)}^{{\mu_0}\otimes \PP_W}$ denotes the completion of $\BBBB(\rd)\otimes \BBBB(\WW_0)$ with respect to the measure ${\mu_0}\otimes \PP_W$.
\begin{definition}[$P_{\mu_0}$-strong solution]\label{yamada.strongSol}
The equation \eqref{SEE} has a $P_{\mu_0}$-strong solution if there exists $F_{\mu_0}\in \tilde{\EEEE}_{\mu_0}$ such that, for $\mu_0$-a.e. $x\in \rd$, $F_{\mu_0}(x,\cdot)$ is $\overline{\BBBB_t(\WW_0)}^{\PP_W}\slash \BBBB_t(\BBB)$-measurable for every $t\in[0,\infty)$, and for any standard $d_1$-dimensional $(\FF_t)$-Brownian motion $W$ on a stochastic basis $(\Omega,\FF,\PP;(\FF_t)_{t\geq 0})$ and any $\FF_0\slash \BBBB(\rd)$-measurable function $\xi: \Omega \to \rd$ with $\PP\circ \xi\inv=\mu_0$, one has that  $(F_{\mu_0}(\xi,W),W)$
	is a $P_{\mu_0}$-weak solution to \eqref{SEE} with $X(0)=\xi$ $\PP$-a.s.
	Here, $\overline{\BBBB_t(\WW_0)}^{\PP_W}$ denotes the completion with respect to $\PP_W$ in $\BBBB(\WW_0)$.
\end{definition}
\begin{definition}[unique $P_{\mu_0}$-strong solution]\label{yamada.uniqStrongSol}
	The equation \eqref{SEE} has a \textit{unique $P_{\mu_0}$-strong solution}, if there exists a function $F_{\mu_0}\in \tilde{\EEEE}_{\mu_0}$ satisfying the adaptedness condition in Definition \ref{yamada.strongSol} and if the following two conditions are satisfied.
	\begin{enumerate}[1.]
		\item For every standard $d_1$-dimensional $(\FF_t)$-Brownian motion $W$ on a stochastic basis $(\Omega,\FF,\PP;(\FF_t)_{t\geq 0})$ and any $\FF_0\slash \BBBB(\rd)$-measurable $\xi:\Omega \to \rd$ with $\law{\xi}=\mu_0$, $(F_{\mu_0}(\xi,W),W)$ is a $P_{\mu_0}$-weak solution.
		\item For any $P_{\mu_0}$-weak solution $(X,W)$ to \eqref{SEE} we have $X=F_{\mu_0}(X(0),W)\ \text{a.s.}$
		\end{enumerate}
\end{definition}
\begin{remark}\label{yamada.remark.weakUniq}
Let $(X,W)$ be a $P_{\mu_0}$-weak solution to \eqref{SEE} on a stochastic basis $(\Omega,\FF,\PP;(\FF_t)_{t\geq 0})$.
	Since $X(0)$ and $W$ are independent, we have
	\begin{align*}
		\PP\circ(X(0),W)\inv=\mu_0\otimes \PP_W.
	\end{align*}
	In particular, $P_{\mu_0}$-weak uniqueness holds for \eqref{SEE} provided there exists a unique $P_{\mu_0}$-strong solution to \eqref{SEE}.
\end{remark}

Let us state the main theorem of this section.
\begin{theorem}[restricted Yamada--Watanabe theorem]
	 \label{SDE.theorem.yamada}
	Let $P_{\mu_0}$, $b$ and $\sigma$ be as above. Then the following statements regarding \eqref{SEE} are equivalent:	\begin{enumerate}[(i)]
		\item \label{yamada.i} There exists a $P_{\mu_0}$-weak solution and $P_{\mu_0}$-pathwise uniqueness holds.
		\item \label{yamada.ii} There exists a unique $P_{\mu_0}$-strong solution.
	\end{enumerate}
\end{theorem}
\begin{proof}[Sketch of proof of Theorem \ref{SDE.theorem.yamada}]
	Since the implication $\eqref{yamada.ii}\implies \eqref{yamada.i}$ is straight forward, let us focus on the other one:
	Let us fix two (not necessarily different) weak solutions $(X^{(i)},W^{(i)})$ on respective stochastic bases {$(\Omega^{(i)},\FF^{(i)},\PP^{(i)};(\FF^{(i)}_t)_{t\geq 0})$}, $i=1,2$.
	As shown in \cite[Lemma E.0.10]{spde}, there exist stochastic kernels $K^{(i)}_{\mu_0}, i=1,2$, from $(\rd\times\WW_0,\BBBB(\rd)\otimes\BBBB(\WW_0))$ to $(\BBB,\BBBB(\BBB))$ such that
	\begin{align*}
		(\PP^{(i)}\circ (X^{(i)}(0),X^{(i)},W^{(i)})\inv)(dx,dw_1,dw) = K^{(i)}_{\mu_0}((x,w),dw_1)\PP_W(dw)\mu_0(dx),\ i=1,2.
	\end{align*}
	
	Now, these two weak solutions can be transferred to a \textit{common} stochastic basis with the \textit{same} Brownian motion and initial value in the following way. Let $t\in [0,\infty)$ and
	\begin{align*}
		\tilde{\Omega} :=&\ \rd\times \BBB\times\BBB\times\WW_0,\\
		\tilde{\FF}_{\mu_0} :=&\ \overline{\BBBB(\rd)\otimes \BBBB(\BBB)\otimes\BBBB(\BBB)\otimes\BBBB(\WW_0)}^{\QQ_{\mu_0}},\\
		\tilde{\FF}^{\mu_0}_t :=&\ \bigcap_{\epsilon>0} \sigma \left(\BBBB(\rd)\otimes \BBBB_{t+\epsilon}(\BBB)\otimes\BBBB_{t+\epsilon}(\BBB)\otimes\BBBB_{t+\epsilon}(\WW_0),\NN_{\mu_0}\right),
	\end{align*}
	where
	$\NN_{\mu_0} := \{N\in\tilde{\FF}_{\mu_0} : \QQ_{\mu_0}(N)=0\}$.
	We define the following measure on $(\tilde{\Omega},\tilde{\FF}_{\mu_0})$:
	\begin{align*}
		\QQ_{\mu_0}(A) := \int_\rd\int_\BBB\int_\BBB\int_{\WW_0} \mathbbm{1}_{A}(z,w_1,w_2,w)\  K_{\mu_0}^{(1)}((z,w),dw_1)K_{\mu_0}^{(2)}((z,w),dw_2)\PP_W(dw)\mu_0(dz).
	\end{align*}
	Note that $(\tilde{\Omega},\tilde{\FF}_{\mu_0},\QQ_{\mu_0};(\tilde{\FF}^{\mu_0}_t)_{t\geq 0})$ is a stochastic basis.
	For $i=0,...,3$, let $\Pi_i$ denote the canonical projection from $\tilde{\Omega}$ onto its $i$-th coordinate.
	We observe that, likewise to \cite[Lemma E.0.11, Lemma E.0.12]{spde} and the techniques therein, $\Pi_3$ is a standard $d_1$-dimensional $(\tilde{\FF}_t^{\mu_0})$-Brownian motion and that $(\Pi_1,\Pi_3)$ and $(\Pi_2,\Pi_3)$ are $P_{\mu_0}$-weak solutions to \eqref{SEE} on  
$(\tilde{\Omega},\tilde{\FF}_{\mu_0},\QQ_{\mu_0};(\tilde{\FF}^{\mu_0}_t)_{t\geq 0})$ such that, for $i=1,2$,
\begin{align*}
	\law{\Pi_0}=\mu_0,\ \law{\Pi_i}=\law{X^{(i)}},\ \ \ and \ \ \ \Pi_1(0)=\Pi_2(0)=\Pi_0\ \ \QQ_{\mu_0}\text{-a.s.}
\end{align*}
By our pathwise uniqueness assumption in \eqref{yamada.i} we therefore deduce
\begin{align*}
	\Pi_1=\Pi_2\ \ \ \QQ_{\mu_0}\text{-a.s.}
\end{align*}
\sloppy In particular, $\law{X^{(1)}} = \law{X^{(2)}}$.
Now, with the same argument as in \cite[Lemma E.0.13]{spde}, there exists a function \mbox{$F_{\mu_0}:\rd\times\WW_0 \to \BBB$} such that
for $(\mu_0\otimes \PP_W)$-a.e. $(x,w) \in \rd\times \WW_0$ 
\begin{align*}
	K_{\mu_0}^{(1)} ((x,w),\cdot)=K_{\mu_0}^{(2)} ((x,w),\cdot)=\delta_{F_{\mu_0}(x,w)},
\end{align*}
(where for $y\in \BBB$, $\delta_y$ denotes the usual Dirac-measure on $(\BBB,\BBBB(\BBB))$)
and, moreover, $F_{\mu_0}$ is $\overline{\BBBB(\rd)\otimes\BBBB_t(\WW_0)}^{\mu_0\otimes \PP_W}\slash \BBBB_t(\BBB)$-measurable for all $t\geq 0$,
where $\overline{\BBBB(\rd)\otimes\BBBB_t(\WW_0)}^{\mu_0\otimes \PP_W}$ denotes the completion of  $\BBBB(\rd)\otimes\BBBB_t(\WW_0)$ with respect to $\mu_0\otimes \PP_W$ on $\BBBB(\rd)\otimes\BBBB(\WW_0)$. This kind of measurability is, in fact, inherited by certain measurability properties of the $K_{\mu_0}^{(i)}$ (for details consult \cite[Lemma E.0.10]{spde}).

We will now show the adaptedness condition for $F_{\mu_0}$ in Definition \ref{yamada.strongSol}, which slightly differs from the one in \cite[Definition E.0.5]{spde}. Therefore, let $\{C_i\}_{i=1}^\infty$ be a countable generator of $\BBBB(\rd)$, i.e. $C_i\subset \rd$, $i\in\LN$, such that $\BBBB(\rd)=\sigma(\{C_i:i\in \LN\})$.

\textbf{Claim:}  For $\mu_0$-a.e. $x\in \rd$
\begin{align}
\{F_{\mu_0} (x,\cdot)\in \pi_q\inv(C_i)\} \in \overline{\BBBB_q(\WW_0)}^{\PP_W}\ \ \ \forall q\in \mathbb{Q}, i \in \LN.
\end{align}
Note that $\BBBB_t(\WW_0) = \sigma(\{\pi_q\inv(C_i):q\in \mathbb{Q}\cap [0,t], i\in \LN\})$. Hence, the adaptedness condition of $F_{\mu_0}$ follows immediately.

\textbf{Proof of Claim:}
	Fix $q\in \mathbb{Q}$ and $i \in \LN$.
	By the measurability of $F_{\mu_0}$, there exist $B_{q,i}\in \BBBB(\rd)\otimes\BBBB_q(\WW_0)$, $\overline{N}_{q,i} \in \BBBB(\rd)\otimes\BBBB(\WW_0)$, and $\underline{N}_{q,i}\in \overline{\BBBB(\rd)\otimes \BBBB_q(\WW_0)}^{\mu_0\otimes\PP_W}$ such that $\underline{N}_{q,i}\subset\overline{N}_{q,i}$,
	\begin{align*}
		(\mu\otimes\PP_W)(\overline{N}_{q,i})=0 \text{\ \ \ and \ \ \ } \{F_{\mu_0} \in \pi_q\inv(C_i)\} = B_{q,i}\cup \underline{N}_{q,i}.
	\end{align*}
	Let $x\in \rd$ and $\text{e}_x: \WW_0 \ni w \mapsto (x,w) \in \rd\times\WW_0$. We define
	$B_x := \text{e}_x\inv(B)$, $\underline{N}^x_{q,i} := \text{e}_x\inv(\underline{N}_{q,i})$, and $\overline{N}^x_{q,i} := \text{e}_x\inv(\overline{N}_{q,i})$. Clearly, $\underline{N}^x_{q,i} \subset \overline{N}_{q,i}^x$. Obviously, $B^x_{q,i} \in \BBBB_q(\WW_0)$, $\overline{N}^x \in \BBBB(\WW_0)$ and
	\begin{align*}
		\{F_{\mu_0}(x,\cdot) \in \pi_q\inv(C_i)\} = B^x_{q,i} \cup \underline{N}^x_{q,i}.
	\end{align*} 
We define $\overline{N}:=\bigcup_{q\in \mathbb{Q},i \in \LN} \overline{N}_{q,i}$ and $\overline{N}^x := \text{e}_x\inv(\overline{N})$. Note that
	\begin{align*}
		0 = (\mu\otimes\PP_W)(\overline{N}) = \int_\rd \PP_W(\overline{N}_x)\ \mu_0(dx).
	\end{align*}
	It follows that, for $\mu_0$-a.e. $x\in \rd$, $\PP_W(\overline{N}^x)=0$. Hence, the claim is proved.\\
	
Now, the rest of the proof is straightforward.
One easily sees (cf. \cite[Lemma E.0.14]{spde}) that
\begin{align*}
	X^{(i)}=F_{\mu_0}(X^{(i)}(0),W^{(i)})\ \ \QQ_{\mu_0}\text{-a.s.},
\end{align*}
i.e. the solutions $X^{(1)},X^{(2)}$ are 'produced' by $F_{\mu_0}$.

Analogous to \cite[E.0.15]{spde}, we observe that $F_{\mu_0}$ 'produces' $P_{\mu_0}$-weak solutions \eqref{SEE} on any given stochastic basis with any given initial random variable distributed as $\mu_0$ and given Brownian motion. The proof of \cite[E.0.15]{spde} directly yields that the law of such a 'produced' solution coincides with the law of $X^{(i)}$ under $\PP^{(i)}$, $i=1,2$. By the pathwise uniqueness assumption in \eqref{yamada.ii}, any $P_{\mu_0}$-weak solution with respect to the same initial random variable and Brownian motion must coincide with the 'produced' solution.
This concludes the sketch of the proof.
\end{proof}
\begin{remark}\label{yamada.remark.finiteTime}
	Note that \cite[Appendix E]{spde} as well as this section can be stated analogously for \eqref{SEE} up to some finite time $T>0$, i.e. for
	\begin{align*}
		dX(t)=b(t,X)dt+\sigma(t,X)dW(t),\ \ t\in [0,T],
	\end{align*}
with coefficients  $b: [0,T] \times C([0,T];\rd) \to \rd$ and $\sigma: [0,T] \times C([0,T];\rd) \to \RR^{d\times d_1}$, which are $\BBBB([0,T])\otimes \BBBB(C([0,T];\rd)) \slash \BBBB(\rd)$ and $\BBBB([0,T])\otimes \BBBB(C([0,T];\rd)) \slash \BBBB(\RR^{d\times d_1})$-measurable, respectively.
In particular, the adaption of the definitions of a $P_{\mu_0}$-weak solution, (unique) $P_{\mu_0}$-strong solution, $P_{\mu_0}$-weak uniqueness, and $P_{\mu_0}$-pathwise uniqueness are straight forward.
\end{remark}
\section{Application of the restricted Yamada--Watanabe theorem to (general) McKean--Vlasov SDEs}\label{section.MVSDE.PME.yamada}
In this section, we will consider general McKean--Vlasov SDEs on $\rd$, which are of the form
\begin{align} \label{MVSDE.general}\tag{MVSDE}
	dX(t) =&\ F(t,X(t),\law{X(t)}) dt + \sigma(t,X(t),\law{X(t)}) dW(t),\ \ t\in[0,T],
\end{align}
where $T\in (0,\infty)$ and
	$F : [0,T] \times \RR^d \times \PPPP(\rd) \to \rd\text{, }
	\sigma : [0,T]  \times \RR^d \times \PPPP(\rd) \to \RR^{d\times d}$
are $\BBBB([0,T])\otimes \BBBB(\rd)\otimes \BBBB(\PPPP(\rd))/\BBBB(\rd)$- and $\BBBB([0,T])\otimes \BBBB(\rd)\otimes \BBBB(\PPPP(\rd))/\BBBB(\RR^{d\times d})$-measurable, respectively.
Let $(\mu_t)_{t\in [0,T]} \subset \PPPP(\rd)$ be a narrowly continuous curve of probability measures.
In the following, we will use the notation $F^\mu(t,x):=F(t,x,\mu_t)$, $(t,x) \in [0,T]\times\rd$. Note that
$F^\mu$ and $\sigma^\mu$ are $\BBBB([0,T])\otimes\BBBB(\rd)\slash \BBBB(\rd)$- and $\BBBB([0,T])\otimes\BBBB(\rd)\slash \BBBB(\RR^{d\times d})$-measurable, respectively.

In the following, let us fix some notation for weak solutions to \eqref{MVSDE.general} and pathwise uniqueness for \eqref{MVSDE.general} among weak solutions with given time marginal laws. 
Therefore, we set
\begin{align*}
		P_{(\mu_t)}:= \{Q \in \PPPP(C([0,T];\rd)): Q \circ \pi_t\inv = \mu_t\  \forall t\in [0,T]\}.
	\end{align*}

We have the following definitions.
\begin{definition}\label{MVSDE.weakSolution}
A tuple $(X,W)=(X(t),W(t))_{t\in [0,T]}$ consisting of two $(\FF_t)$-adapted $\RR^d$-valued stochastic processes
 on some given stochastic basis $(\Omega,\FF,\PP;(\FF_t)_{t\in [0,T]})$
is called a $P_{(\mu_t)}$-weak solution to \eqref{MVSDE.general} if
$W$ is a standard $d$-dimensional $(\FF_t)$-Brownian motion, and
\begin{enumerate}[(i)]
\item $\PP\left(\int_0^T |F(t,X(t),\law{X(t)})| + |\sigma(t,X(t),\law{X(t)})|^2 dt<\infty\right)=1,$
\item the following equality holds $\PP$-a.s.:
		 \begin{align*}
		 	X(t)=X(0) + \int_0^t F(s,X(s),\law{X(s)}) ds + \int_0^t \sigma(s,X(s),\law{X(s)}) dW(s)\ \forall t\in[0,T],
		 \end{align*}
\item $\PP\circ (X(t))\inv = \mu_t$, for all $t\in [0,T]$.
\end{enumerate}
\end{definition}
\begin{definition}\label{definition.restrPU}
	We say that $P_{(\mu_t)}$-pathwise uniqueness holds for \eqref{MVSDE.general}, if for any two $P_{(\mu_t)}$-weak solutions $(X,W)$, $(Y,W)$ on a common stochastic basis 
	$(\Omega,\FF,\PP;(\FF_t)_{t\in [0,T]})$ with a common standard $d$-dimensional $(\FF_t)$-Brownian motion $W$, 
	$$\text{$X(0)=Y(0)$ $\PP$-a.s. implies $X(t)=Y(t)$ for all $t\in [0,T]\ \PP$-a.s.}$$
\end{definition}

Let us note that $(X,W)$ is a $P_{(\mu_t)}$-weak solution to \eqref{MVSDE.general} if and only if it is a $P_{(\mu_t)}$-weak solution to the SDE(!)
\begin{align}\label{MVSDE.fixed.u}\tag{$\text{SDE}_\mu$}
	dX(t)&=F^\mu(t,X(t))dt + \sigma^\mu(t,X(t))dW(t),\ \ t\in [0,T],
\end{align}
and, obviously, $P_{(\mu_t)}$-pathwise uniqueness holds for \eqref{MVSDE.general} if and only if it holds for \eqref{MVSDE.fixed.u}.
Therefore, also the concepts of a (unique) $P_{(\mu_t)}$-strong solution is the same for \eqref{MVSDE.general} and \eqref{MVSDE.fixed.u}.

This leads to the following application of the restricted Yamada--Watanabe theorem (Theorem \ref{SDE.theorem.yamada}) to \eqref{MVSDE.general}.
\begin{theorem}
\label{MVSDE.restrYamada}
Let $(\mu_t)_{t\in [0,T]}\subset \PPPP_0(\rd)$ be a narrowly continuous curve of probability measures.
The following statements regarding \eqref{MVSDE.general} are equivalent. 
\begin{enumerate}[(i)]
	\item There exists a $P_{(\mu_t)}$-weak solution and $P_{(\mu_t)}$-pathwise uniqueness holds.
	\item There exists a unique $P_{(\mu_t)}$-strong solution.
\end{enumerate}
\end{theorem}

Let us finally note that \eqref{MVSDE} can be considered as an equation in the general form \eqref{MVSDE.general} as the following remark conveys.
\begin{remark}\label{Remark_MVSDE.PME_is_MVSDE}
	The coefficients of \eqref{MVSDE} fulfill the measurability conditions of the coefficients of \eqref{MVSDE.general} in the following sense.
	Consider the time-homogenous coefficients 
	$F:\rd\times \PPPP(\rd) \to \rd$, $\sigma: \rd \times \PPPP(\rd) \to \RR^{d\times d}$ defined via
	\begin{align*}
			F(x,\nu):= E(x)b(v_a(x)),\ \sigma(x,\nu):= \sqrt{2a\left(v_a(x)\right)}\mathbbm{1}_{d\times d},
	\end{align*}
where $x \in\rd, \nu \in \PPPP(\rd)$ and $v_a$ is the following version of the $\lambda^d$-a.e. uniquely determined density of the absolutely continuous part of the probability measure $\nu$ given by the Besicovitch derivation theorem
\begin{align*}
	v_a(x) := \left\{
	\begin{array}{ll}
	\lim_{r\to \infty} \frac{\nu(B_{r}(x))}{\lambda^d(B_{r}(0))}, & x \in E_{\nu}, \\
	0, & x\in E_{\nu}^\complement, \\
	\end{array}
	\right. \text{\ \ \ \ for all $\nu \in \PPPP(\rd)$},
\end{align*}
where $E_\nu := \left\{ x \in \rd : \exists \lim_{r\to \infty} \frac{\nu(B_{r}(x))}{\lambda^d(B_{r}(0))} \in \RR \right\} \in \BBBB(\rd)$ and $\lambda^d(E_\nu^\complement)=0$ (cf. \cite[Theorem 2.22]{ambrosioBV}).
Hence, it is easy to see that $$\rd\times \PPPP(\rd) \ni (x,\nu) \mapsto v_a(x) \in [0,\infty)$$ is $\BBBB(\rd)\otimes\BBBB(\PPPP(\rd))\slash \BBBB(\RR)$ measurable. 
In the following, we will always consider this version of the absolutely continuous part of a probability measure.
\end{remark}

Now, we are well equipped to translate Theorem \ref{MVSDE.restrYamada} into action.
\section{Strong solvability of \eqref{MVSDE}}\label{section.MVSDE.PME.strong}
\subsection{The procedure and the main result}\label{section.MVSDE.PME.Procedure}
Our overall goal is to apply Theorem \ref{MVSDE.restrYamada}, which will enable us to show that there exists a strong solution to \eqref{MVSDE} (see Theorem \ref{MVSDE.restrYamada}). In order to achieve this, we will do the following steps.
\begin{enumerate}
	\item \label{procedure.1} We will use the recent result \cite{barbu2020solutions} (and the techniques of \cite{barbu2021evolution}) for \eqref{PME}, in order to guarantee the existence of a probability solution $u$ with sufficient Sobolev-regularity under the conditions \eqref{condition.beta.general}-\eqref{condition.b} (see Theorem \ref{PME.theorem.existence}).
	\item \label{procedure.2} We will apply the superposition principle procedure for McKean--Vlasov SDEs from \cite[Section 2]{barbu2019nonlinear} in combination with the result of Step \ref{procedure.1} in order to obtain a weak solution to \eqref{MVSDE} with time marginal law densities $u$ (see Theorem \ref{MVSDE.PME.theorem.weakExistence}).
	\item \label{procedure.3} Afterwards, we will prove pathwise uniqueness for \eqref{MVSDE}  among weak solutions with time marginal law densities $u$ via a pathwise uniqueness result for SDEs (see Theorem \ref{SDE.theorem.pathwiseuniqueness}) and Step \ref{procedure.1} in Theorem \ref{MVSDE.PME.theorem.pathwiseuniqueness}.
\end{enumerate}
For the ease of notation, we set $P_{(u_t)}:=P_{(\mu_t)}$, whenever $(\mu_t)_{t\in [0,T]}$ is a narrowly continuous curve of probability measures with $\mu_t = u_t(x)dx, u_t \in \PPPP_0(\rd)$, $t\in [0,T]$.

The steps will be carried out in the subsequent subsections. Combining the results of the steps with Theorem \ref{MVSDE.restrYamada} yield the main result of this section and paper.
\begin{theorem}[main result]\label{MVSDE.PME.theorem.strongExistence}
Let $d\neq 2$.
 Assume that conditions \eqref{condition.beta.general}-\eqref{condition.b} and \eqref{condition.a} (see below) are fulfilled and that $u_0 \in \PPPP_0(\rd)\cap L^\infty(\rd)$.
  Then, \eqref{MVSDE} has a unique $P_{(u_t)}$-strong solution,
  where
  $u$ is the constructed probability solution to \eqref{PME} provided by Theorem \ref{PME.theorem.existence}.
\end{theorem}
The proof of Theorem \ref{MVSDE.PME.theorem.strongExistence} will be postponed to the end of this section.
\subsection{Existence of a bounded Sobolev-regular probability solution $u$ to \eqref{PME}}\label{section.PME}
As described in the procedure in Section \ref{section.MVSDE.PME.Procedure}, the first step is to conclude the existence of a sufficiently regular solution to \eqref{PME} from \cite{barbu2020solutions} under the assumptions \eqref{condition.beta.general}-\eqref{condition.b}.
Combining \cite[Theorem 2.2]{barbu2020solutions} with the techniques of \cite{barbu2021evolution}, Barbu and Röckner showed that,
under more general assumptions on the coefficients than we require in \eqref{condition.beta.general}-\eqref{condition.b}, there exists a unique mild solution to \eqref{PME}, interpreted as a Cauchy problem driven by an $m$-accretive operator. This solution is also an integrable and bounded Schwartz-distributional solution to \eqref{PME} if $u_0$ is integrable and bounded as well.
If, in addition, $\beta$ is non-degenerate, i.e. $\beta' \geq \gamma_0>0$, this \textit{specific} solution can be proved to have certain desirable Sobolev-regularity under our conditions.
Our approach relies on exactly this regularity when proving the pathwise uniqueness result for \eqref{MVSDE} (cf. Theorem \ref{MVSDE.PME.theorem.pathwiseuniqueness}).

We have the following
\begin{theorem}[probability solution to \eqref{PME}]\label{PME.theorem.existence}
	Let $d\neq 2$
	and $u_0 \in\PPPP_0(\rd)\cap L^\infty(\rd)$. Under the assumptions \eqref{condition.beta.general}-\eqref{condition.b}, there exists a probability solution $u$ to \eqref{PME} such that
	\begin{align}\label{pu_u_reg_detail}
		u\in L^2([0,T];W^{1,2}(\rd))\cap L^\infty([0,T]\times\rd).
	\end{align}
\end{theorem}
\begin{proof}
By \cite[Theorem 2.2]{barbu2020solutions}, we know that there exists a Schwartz-distributional solution $u$ to \eqref{PME} with ${u \in C([0,T];L^1(\rd))\cap L^\infty([0,T]\times\rd)}$, and with the property that 
$u_0 \in \PPPP_0(\rd)$ implies  $u_t\in \PPPP_0(\rd)$, for all $t\in [0,T]$.
Here, we should note that the authors require $\beta\in C^2(\rd)$.
However, due to \eqref{condition.beta.monotone} and $\divv D \in L^2(\rd)+L^\infty(\rd)$, this condition can be relaxed to $\beta \in C^1(\RR)$; this works analogous to \cite[p. 20, proof of (2.6)]{barbu2021evolution}.
Further, using the technique in \cite{barbu2021evolution}, our assumptions imply that $u \in L^2([0,T];W^{1,2}(\rd))$. For an elaborate proof of these facts, see \cite{grube2022thesis}.
\end{proof}
\subsection{The existence of a $P_{(\mu_t)}$-weak solution to \eqref{MVSDE}}\label{section.MVSDE.PME.weakExistence}
The second step of the procedure in Section \ref{section.MVSDE.PME.Procedure}
is to show the existence of a weak solution to \eqref{MVSDE}.

The following theorem is a variant of \cite[Theorem 6.1 (a)]{barbu2020solutions} and is based on a superposition principle procedure for McKean--Vlasov SDEs as described in \cite[Section 2]{barbu2019nonlinear}, which generalises the procedure in \cite[Section 2]{barbu2018Prob}).
\begin{theorem}[$P_{(\mu_t)}$-weak solution]\label{MVSDE.PME.theorem.weakExistence}Let $d\neq 2$
and $u_0 \in\PPPP_0(\rd)\cap L^\infty(\rd)$. Assume that conditions \eqref{condition.beta.general}-\eqref{condition.b} are fulfilled.
	Then, there exists a $P_{(u_t)}$-weak solution (X,W) to \eqref{MVSDE},
	 where
	$u$ is the probability solution provided by Theorem \ref{PME.theorem.existence}.
\end{theorem}
\begin{proof}
It is clear that conditions \eqref{condition.beta.general} and \eqref{condition.b} imply that $b$ and $a$ are continuous. By Theorem \ref{PME.theorem.existence}, we have that, in particular, $u\in L^\infty([0,T]\times\rd)$. Hence, $a(u)\in L^\infty([0,T]\times\rd)$, and, using \eqref{condition.D}, $E b(u) \in L^\infty([0,T]\times\rd;\rd)$. This yields
	\begin{align*}
		\int_0^T \int_\rd \left[|a(u(t,x))|+|E(x)b(u(t,x))|\right] u(t,x)dxdt<\infty.
	\end{align*}
	This enables us to use the superposition principle procedure for McKean--Vlasov SDEs in \cite[Section 2]{barbu2019nonlinear} which provides us with a $P_{(u_t)}$-weak solution $(X,W)$ to \eqref{MVSDE}.
	This finishes the proof.
\end{proof}
\subsection{$P_{(\mu_t)}$-pathwise uniqueness for \eqref{MVSDE}}\label{section.MVSDE.PME.pathwiseUniqueness}
The third step of the procedure in Section \ref{section.MVSDE.PME.Procedure} is to show $P_{(u_t)}$-pathwise uniqueness for \eqref{MVSDE}, where $u$ is the probability solution to \eqref{PME} provided by Theorem \ref{PME.theorem.existence}.
As explained in the beginning of Section \ref{section.MVSDE.PME.strong}, showing $P_{(u_t)}$-pathwise uniqueness for \eqref{MVSDE} is the same as showing $P_{(u_t)}$-pathwise uniqueness for \eqref{MVSDE.fixed.u}, where $\mu = (u_tdx)_{t\in [0,T]}$.
Since the coefficients $Eb(u)$ and $\sqrt{2a(u)}$ are not continuous in the spacial variable, we will recall a pathwise uniqueness result for SDEs with time-dependent Sobolev-coefficients from \cite{roeckner2010weakuniqueness} in Subsection \ref{section.SDE.pathwiseUniqueness}.
In  Subsection \ref{section.PU.app}, we will then apply this result to show $P_{(\mu_t)}$-pathwise uniqueness for \eqref{MVSDE}.
\subsubsection{A pathwise uniqueness result for SDEs with bounded Sobolev-regular coefficients}\label{section.SDE.pathwiseUniqueness}

There are a lot of pathwise uniqueness results for SDEs across the literature, but numerous classical as well as recent results on this topic require the diffusion coefficient to be continuous in the spacial variable. 
For example, in \cite{zhang2011singular} and \cite{ling2019sdes} (see also the classic result \cite{Veretennikov_1981}), the authors consider a (singular) drift coefficient in $L^q([0,T];L^p(\rd))$, $\nicefrac{d}{p}+\nicefrac{2}{q}<1$, where $p,q \in (2,\infty)$, and an elliptic diffusion coefficient, whose distributional derivative is in $L^q([0,T];L^p(\rd))$, for the same choice of $p$ and $q$ as before. Additionally, the diffusion coefficient is considered to be uniformly continuous in the $x$-variable locally uniformly in time. 

To the best of our knowledge, the best pathwise uniqueness results for SDEs, with no a priori continuity assumption on the diffusion term, can be obtained through \cite{roeckner2010weakuniqueness} and \cite{champagnat2018};
these works require Sobolev-regularity of the coefficients in the $x$-variable, but no a priori continuity property. Here, we will just focus on the result obtained in \cite{roeckner2010weakuniqueness}, as it allows the coefficients to only have local Sobolev-regularity in the spatial variable.

In this subsection, we will provide the reader with a simple modification of a restricted pathwise uniqueness result for SDEs in the proof of \cite[Theorem 1.1]{roeckner2010weakuniqueness}.
The strength of this result is that there is a trade-off between the regularity of the densities of the time marginal laws of a solution process and the regularity of the coefficients of the equation.
Further, the estimate \cite[Lemma A.3]{crippa2008estimates} and \cite[Lemma A.2]{crippa2008estimates}, involving the (local) Hardy-Littlewood maximal function, shows that sufficient Sobolev-regular coefficients of \eqref{SDE} (see below) satisfy \eqref{SDE.theorem.pathwiseuniqueness.oneSidedLipschitz}, see Remark \ref{SDE.remark.pathwiseuniqueness}.
Here, we also refer to \cite{champagnat2018}, where the authors developed an interesting modification of this estimate (see \cite[Lemma 3.2]{champagnat2018}), which turned out to be very useful when showing a restricted pathwise uniqueness result in the critical case when the drift coefficient of an SDE has $L^1([0,T];W^{1,1}(\rd))$-regularity (cf. \cite[Theorem 1.1]{champagnat2018}).

Let us consider the following stochastic differential equation
\begin{align}\label{SDE}\tag{SDE}
	dX(t)&= \boldsymbol{F}(t,X(t))dt + \boldsymbol{\sigma}(t,X(t))dW(t),\ \  t\in [0,T]\\
	X(0)&=\xi.\notag
\end{align} 
where
$\boldsymbol{F} : [0,T] \times \RR^d \to \RR^d\text{, }
	\boldsymbol{\sigma} : [0,T]  \times \RR^d \to \RR^{d\times d}$
are $\BBBB([0,T])\otimes\BBBB(\rd)/\BBBB(\rd)$- and $\BBBB([0,T])\otimes\BBBB(\rd)/\BBBB(\RR^{d\times d})$-measurable functions, respectively; the initial condition $\xi$ and the $d$-dimensional Brownian motion $W$ are considered to be analogous to those introduced in the beginning of this work.

We have the following
\begin{theorem}[restricted pathwise uniqueness for \eqref{SDE}]\label{SDE.theorem.pathwiseuniqueness}
Let $\boldsymbol{F},\boldsymbol{\sigma} \in L^\infty([0,T]\times\rd)$. Fix $p, q, p', q'\in [1,\infty]$, such that
	$\nicefrac{1}{p}+\nicefrac{1}{p'}=\nicefrac{1}{q}+\nicefrac{1}{q'}=1.$
	Let $(X,W),(Y,W)$ be two (usual) weak solutions to \eqref{SDE} up to time $T$ on a common stochastic basis $(\Omega,\FF,\PP; (\FF_t)_{t\in [0,T]})$ with $X(0)=Y(0)$ $\PP$-a.s., such that
	\begin{align}\label{SDE.theorem.pathwiseuniqueness.lawintegrability}
		\frac{d\law{X(\cdot)}}{dx},\frac{d\law{Y(\cdot)}}{dx} \in L^{q'}([0,T];L^{p'}_{loc}(\rd)).
	\end{align}
	If for any radius $R>0$, there exists a function $f_R \in L^q([0,T];L^p(B_R(0))$, such that for almost every $(t,x,y) \in [0,T]\times B_R(0)\times B_R(0)$
	\begin{align}\label{SDE.theorem.pathwiseuniqueness.oneSidedLipschitz}
		2\scalarproduct[\rd]{x-y}{\boldsymbol{F}(t,x)-\boldsymbol{F}(t,y)}+|\boldsymbol{\sigma}(t,x)-\boldsymbol{\sigma}(t,y)|^2 \leq (f_R(t,x)+f_R(t,y))\cdot |x-y|^2.
	\end{align}
	Then, $\sup_{t\in [0,T]}|X(t)-Y(t)|=0$.
\end{theorem}
\begin{proof}
	The proof is essentially contained in the proof of \cite[Theorem 1.1]{roeckner2010weakuniqueness}. However, since we allow different integrability in space and time for the time marginal law densities in \eqref{SDE.theorem.pathwiseuniqueness.lawintegrability} and for $f_R$, we need to separately apply Hölder-estimates in the proof of \cite[Theorem 1.1]{roeckner2010weakuniqueness}. We omit the details here, since the adaption of the proof is straightforward.
\end{proof}
The following remark will be useful when checking \eqref{SDE.theorem.pathwiseuniqueness.oneSidedLipschitz} in applications (see, e.g. Theorem \ref{MVSDE.PME.theorem.pathwiseuniqueness}).
\begin{remark}\label{SDE.remark.pathwiseuniqueness}\begin{enumerate}[\label=(a)]
	\item 
In the case $q=p=1$, the proof of \cite[Theorem 1.1]{roeckner2010weakuniqueness} allows to replace \eqref{SDE.theorem.pathwiseuniqueness.lawintegrability} by
\begin{align*}
	\frac{d\law{X(\cdot)}}{dx},\frac{d\law{Y(\cdot)}}{dx} \in L^{\infty}_{loc}([0,T]\times\rd)).
\end{align*} 
Similarly, in the case $q=p=\infty$, the regularity assumption on $f_R$ appearing in \eqref{SDE.theorem.pathwiseuniqueness.oneSidedLipschitz} can be replaced by
\begin{align*}
	f_R \in L^{\infty}([0,T]\times B_R(0)).
\end{align*}
\item In the case $q=p=1$, \eqref{SDE.theorem.pathwiseuniqueness.oneSidedLipschitz} is satisfied if, for some $\epsilon>0$,
	\begin{align*}
		\boldsymbol{F} \in L^{1}([0,T];W_{loc}^{1,1+\epsilon}(\rd;\rd)),\ \boldsymbol{\sigma} \in L^{2}([0,T];W_{loc}^{1,  2}(\rd;\RR^{d\times d})).
	\end{align*}
	For details see [Remark 1.2,RO10].
	\end{enumerate}
\end{remark}

\subsubsection{Application to \eqref{MVSDE}}\label{section.PU.app}
In this subsection, we will apply the pathwise uniqueness result for SDEs from the previous section to \eqref{MVSDE}.
We impose the following additional assumption on $a$ and $E$, respectively.
\begin{align}\label{condition.a}\tag{\text{v}} 
	a \text{ is locally Lipschitz continuous, } \nabla E \in L^2_{loc}(\rd;\RR^{d\times d}).
\end{align}
Note that if $\beta\in C^2(\rd)$, then $a \in C^1(\rd)$ with $a'(0)=\frac{1}{2}\beta''(0)$ and the first part of \eqref{condition.a} is automatically satisfied.

We have the following

\begin{theorem}[$P_{(u_t)}$-pathwise uniqueness]\label{MVSDE.PME.theorem.pathwiseuniqueness}
	Let $d\neq 2$
	and $u_0 \in\PPPP_0(\rd)\cap L^\infty(\rd)$. Assume that the conditions \eqref{condition.beta.general}-\eqref{condition.b} and \eqref{condition.a} are fulfilled. 
	Let $(X,W),(Y,W)$ be $P_{(u_t)}$-weak solutions to \eqref{MVSDE}
	 on the same filtered probability space $(\Omega, \FF, \PP; (\FF_t)_{t\in[0,T]})$, with the same Brownian motion $W$, and $X(0) = Y(0)$ $\PP$-a.s.,
	where
	$u$ is the probability solution provided by Theorem \ref{PME.theorem.existence}.

	Then, $\sup_{t \in [0,T]} |X(t) - Y(t)| = 0$ $\PP$-a.s.
\end{theorem}
\begin{proof}[Proof of Theorem \ref{MVSDE.PME.theorem.pathwiseuniqueness}]
Let $u$ denote the probability solution provided by Theorem \ref{PME.theorem.existence}.
As explained before, we exactly need to show that $P_{(u_t)}$-pathwise uniqueness holds for the SDE
\begin{align*}
	dX(t) &= E(X(t))b(u_t(X(t))dt + \sqrt{2a(u_t(X(t)))}dW(t),\ \ t\in [0,T],\\
	X(0)&=\xi.
\end{align*}
We will now check the conditions in Theorem \ref{SDE.theorem.pathwiseuniqueness}.
By Remark \ref{SDE.remark.pathwiseuniqueness}, these will be implied by the following conditions together with \eqref{pu_u_reg_detail}.
	\begin{enumerate}[(a)]
		\item \label{pu_coeff_Linfty} $Eb(u) \in L^\infty([0,T]\times\rd;\rd), a(u) \in L^\infty([0,T]\times\rd)$,
		\item \label{pu_diffu_sobolev}$\nabla(\sqrt{a(u)}) \in L^2([0,T];L^2(\rd))$,
		\item \label{pu_drift_sobolev} $\nabla(Eb(u)) \in L^2([0,T];L^2_{loc}(\rd;\RR^{d\times d}))$.
	\end{enumerate}
	
	Clearly, \eqref{pu_coeff_Linfty} is satisfied, since $b$ and $a$ are continuous
	and $E$ and $u$ are bounded.
	Let us now show condition \eqref{pu_diffu_sobolev}. 
	Due to \eqref{condition.a.nondegenerate} and \eqref{condition.a}, $\sqrt{a}$ is locally Lipschitz continuous. Let $L>0$ be the Lipschitz constant of $\sqrt{a}$ on the interval $[-\norm[L^\infty]{u},\norm[L^\infty]{u}]$. By \eqref{pu_u_reg_detail}, $\sqrt{a(u_t)})$ has a weak gradient satisfying
	\begin{align*}
		|\nabla\sqrt{a(u_t)}|\leq L|\nabla u_t| \text{ a.e.},
	\end{align*}
	for almost every $t\in [0,T]$ (cf. \cite[Theorem 2.1.11.]{ziemer89}).
	Using \eqref{pu_u_reg_detail}, we see that condition \eqref{pu_diffu_sobolev} is fulfilled.
	
	Let us now turn to condition \eqref{pu_drift_sobolev}.
	Analogous to the proof of \eqref{pu_diffu_sobolev}, condition \eqref{pu_u_reg_detail} and $b\in C^1(\RR)$ imply that $\nabla (b(u))\in L^2([0,T];L^2(\rd))$.
	Having in mind \eqref{condition.D}, \eqref{condition.b} and \eqref{condition.a}, we estimate using the product rule  for Sobolev functions (see, e.g. \cite[Theorem 4.4]{evans_MIT})
	\begin{align*}
		&\norm[{L^2([0,T];L^2(B_R(0);\RR^{d\times d}))}]{\nabla(Eb(u))}  = \norm[{L^2([0,T];L^2(B_R(0);\RR^{d\times d}))}]{(\nabla E)b(u) + (E\otimes \nabla (b(u)))} \\
		&\ \ \ \ \ \leq T^{\nicefrac{1}{2}} \norm[{L^\infty([0,T]\times\rd)}]{b(u)}\norm[{L^2(B_R(0);\RR^{d\times d})}]{\nabla E}\\   
		&\ \ \ \ \ \ \ \ +d \norm[{L^\infty([0,T]\times\rd;\rd)}]{E}\norm[{L^2([0,T];L^2(\rd;\rd))}]{\nabla (b(u))} < \infty,
	\end{align*}
		where $R>0$ is arbitrary and $\otimes$ denotes the usual dyadic product in this (and only this) calculation.
	This finishes the proof.
\end{proof}
The proof of the main result, Theorem \ref{MVSDE.PME.theorem.strongExistence}, is straight forward to conclude.
\begin{proof}[Proof of Theorem \ref{MVSDE.PME.theorem.strongExistence}]
	On the one hand, by Theorem \ref{MVSDE.PME.theorem.weakExistence}, there exists a $P_{(u_t)}$-weak solution to \eqref{MVSDE}. On the other hand, Theorem \ref{MVSDE.PME.theorem.pathwiseuniqueness} shows that $P_{(u_t)}$-pathwise uniqueness holds for \eqref{MVSDE}. Now, the assertion follows from Theorem \ref{MVSDE.restrYamada}. 
\end{proof}
\begin{remark}
	In fact, we can relax the integrability condition on $E$ in \eqref{condition.a} to $\nabla E \in L^1_{loc}(\rd;\RR^{d\times d})$. Indeed, one can use the stopping time technique of the proof of \cite[Theorem 1.1]{roeckner2010weakuniqueness} and the general technique of the proof of \cite[Theorem 1.1 (ii)]{champagnat2018} to prove, in particular, a pathwise uniqueness result for \eqref{SDE} with coefficients $\boldsymbol{F} \in L^\infty([0,T]\times\rd;\rd)\cap L^1([0,T];W^{1,1}_{loc}(\rd;\rd))$ and $\boldsymbol{\sigma} \in L^\infty([0,T]\times\rd;\RR^{d\times d})\cap L^2([0,T];W^{1,2}_{loc}(\rd;\RR^{d\times d}))$ among weak solutions with bounded time marginal law densities. The details are carried out in the author's thesis \cite{grube2022thesis}.
\end{remark}
\begin{remark}
	With a similar technique as in this work, we are able to prove the existence of a strong solution to the degenerate McKean--Vlasov SDE associated to the classical porous medium equation in one dimension
	\begin{align*}
		\partial_t u = \partial^2_x (|u|^{m-1}u),\ \  t\in [0,T],
		\left.u\right|_{t=0} = u_0 \in L^\infty(\RR)\cap \PPPP_0(\RR),
	\end{align*}
	where $m> 3$. From \cite[Theorem 1.2]{gess17}, it is known that there exists a unique entropy solution satisfying $u^\frac{m-1}{2} \in L^2([0,T];W^{\frac{1}{2},1}(\RR))$. If $u_0$ is chosen as above, $u$ is a bounded probability solution to this equation. As above, the existence of a $P_{(u_t)}$-weak solution is argumented by the superposition principle procedure for MVSDEs. Employing the one dimensional pathwise uniqueness result \cite[Theorem 1.2]{champagnat2018}, we can show $P_{(u_t)}$-pathwise uniqueness to the corresponding MVSDE. Using a similar procedure to the one in Section \ref{section.MVSDE.PME.Procedure}, we obtain a unique $P_{(u_t)}$-strong solution.
	To the best of the author's knowledge, this result seems to be new.
	Note that in \cite{Benachour96} the authors proved the existence of a unique strong solution in the case $d=1$ under the stronger assumption that the initial distribution density is given by $u_0 \in (\PPPP_0\cap C^1_b)(\RR)$ with the property that $u_0$ is always strictly positive.
\end{remark}
\textbf{Acknowledgements:}
I would like to express my utmost gratitude towards my supervisor Prof. Dr. Michael Röckner, who pointed out this interesting topic to me and with whom I shared various fruitful discussions.
Moreover, I gratefully acknowledge the support by the German Research Foundation (DFG) through the IRTG 2235.
\bibliographystyle{alpha}
\bibliography{Bibliography}

\begin{thebibliography}{BCRV96}

\bibitem[AFP00]{ambrosioBV}
Luigi Ambrosio, Nicola Fusco, and Diego Pallara.
\newblock {\em {F}unctions of bounded variation and free discontinuity
  problems}.
\newblock Oxford Mathematical Monographs. The Clarendon Press, Oxford
  University Press, New York, 2000.

\bibitem[BCRV96]{Benachour96}
Said Benachour, P.~Chassaing, Bernard Roynette, and Pierre Vallois.
\newblock Processus associ\'es \`a l'\'equation des milieux poreux.
\newblock {\em Annali della Scuola Normale Superiore di Pisa - Classe di
  Scienze}, 4e s{\'e}rie, 23(4):793--832, 1996.

\bibitem[BJ19]{bossy2019moderated}
Mireille Bossy and Jean-Fran\c{c}ois Jabir.
\newblock {O}n the wellposedness of some {M}c{K}ean models with moderated or
  singular diffusion coefficient.
\newblock In {\em Frontiers in stochastic analysis---{BSDE}s, {SPDE}s and their
  applications}, volume 289 of {\em Springer Proc. Math. Stat.}, pages 43--87.
  Springer, Cham, 2019.

\bibitem[BR18]{barbu2018Prob}
Viorel Barbu and Michael R\"{o}ckner.
\newblock {P}robabilistic representation for solutions to nonlinear
  {F}okker-{P}lanck equations.
\newblock {\em SIAM J. Math. Anal.}, 50(4):4246--4260, 2018.

\bibitem[BR20]{barbu2019nonlinear}
Viorel Barbu and Michael R\"{o}ckner.
\newblock {F}om nonlinear {F}okker-{P}lanck equations to solutions of
  distribution dependent {SDE}.
\newblock {\em Ann. Probab.}, 48(4):1902--1920, 2020.

\bibitem[BR21a]{barbu2020solutions}
Viorel Barbu and Michael R\"{o}ckner.
\newblock Solutions for nonlinear {F}okker-{P}lanck equations with measures as
  initial data and {M}c{K}ean-{V}lasov equations.
\newblock {\em J. Funct. Anal.}, 280(7):108926, 2021.

\bibitem[BR21b]{barbu2021evolution}
Viorel Barbu and Michael Röckner.
\newblock {T}he evolution to equilibrium of solutions to nonlinear
  {F}okker--{P}lanck equation.
\newblock {\em To appear in: Indiana Univ. Math. J.}, 2021.
\newblock arXiv:1904.08291v8.

\bibitem[BRS21]{bogachev2021super}
Vladimir~I. Bogachev, Michael R\"{o}ckner, and Stanislav~V. Shaposhnikov.
\newblock On the {A}mbrosio-{F}igalli-{T}revisan {S}uperposition {P}rinciple
  for {P}robability {S}olutions to {F}okker-{P}lanck-{K}olmogorov {E}quations.
\newblock {\em J. Dynam. Differential Equations}, 33(2):715--739, 2021.

\bibitem[CD18]{delarue2018mckean}
Ren\'{e} Carmona and Fran\c{c}ois Delarue.
\newblock {\em Probabilistic theory of mean field games with applications.
  {I}\&{II}}.
\newblock Probability Theory and Stochastic Modelling. Springer, Cham, 2018.

\bibitem[CDL08]{crippa2008estimates}
Gianluca Crippa and Camillo De~Lellis.
\newblock Estimates and regularity results for the {D}i{P}erna-{L}ions flow.
\newblock {\em J. Reine Angew. Math.}, 616:15--46, 2008.

\bibitem[CJ18]{champagnat2018}
Nicolas Champagnat and Pierre-Emmanuel Jabin.
\newblock Strong solutions to stochastic differential equations with rough
  coefficients.
\newblock {\em Ann. Probab.}, 46(3):1498--1541, 05 2018.

\bibitem[EG15]{evans_MIT}
Lawrence~C. Evans and Ronald~F. Gariepy.
\newblock {\em Measure theory and fine properties of functions}.
\newblock Textbooks in Mathematics. CRC Press, Boca Raton, FL, revised edition,
  2015.

\bibitem[Fig08]{figalli2008}
Alessio Figalli.
\newblock Existence and uniqueness of martingale solutions for {SDE}s with
  rough or degenerate coefficients.
\newblock {\em J. Funct. Anal.}, 254(1):109--153, 2008.

\bibitem[Fun84]{funaki1984}
Tadahisa Funaki.
\newblock A certain class of diffusion processes associated with nonlinear
  parabolic equations.
\newblock {\em Probability Theory and Related Fields}, 67(3):331--348,
  September 1984.

\bibitem[Gru22]{grube2022thesis}
Sebastian Grube.
\newblock Dissertation.
\newblock 2022.

\bibitem[GST20]{gess17}
Benjamin Gess, Jonas Sauer, and Eitan Tadmor.
\newblock Optimal regularity in time and space for the porous medium equation.
\newblock {\em Anal. PDE}, 13(8):2441--2480, 2020.

\bibitem[HW19]{huang2019singular}
Xing Huang and Feng-Yu Wang.
\newblock Distribution dependent {SDE}s with singular coefficients.
\newblock {\em Stochastic Process. Appl.}, 129(11):4747--4770, 2019.

\bibitem[HW20a]{huang2020mckeanvlasov}
Xing Huang and Feng-Yu Wang.
\newblock {M}c{K}ean-{V}lasov {SDE}s with drifts discontinuous under
  {W}asserstein distance, 2020.

\bibitem[HW20b]{huang2020wellposedness}
Xing Huang and Feng-Yu Wang.
\newblock Well-posedness for singular {M}c{K}ean-{V}lasov stochastic
  differential equations, 2020.

\bibitem[Kur07]{kurtz2007yamada}
Thomas~G. Kurtz.
\newblock The {Y}amada-{W}atanabe-{E}ngelbert theorem for general stochastic
  equations and inequalities.
\newblock {\em Electron. J. Probab.}, 12:951--965, 2007.

\bibitem[Kur14]{kurtz2014yamada}
Thomas~G. Kurtz.
\newblock Weak and strong solutions of general stochastic models.
\newblock {\em Electron. Commun. Probab.}, 19:no. 58, 16, 2014.

\bibitem[LR15]{spde}
Wei Liu and Michael R\"{o}ckner.
\newblock {\em {S}tochastic {P}artial {D}ifferential {E}quations: {A}n
  {I}ntroduction}.
\newblock Universitext. Springer, Cham, 2015.

\bibitem[LRZ19]{ling2019sdes}
Chengcheng Ling, Michael Röckner, and Xiangchan Zhu.
\newblock {SDE}s with singular drifts and multiplicative noise on general
  space-time domains, 2019.
\newblock arXiv:1910.03989.

\bibitem[Mj66]{mckean1966class}
H.~P. McKean~jr.
\newblock A class of {M}arkov processes associated with nonlinear parabolic
  equations.
\newblock {\em PNAS}, 56(6):1907--1911, 1966.

\bibitem[MV20]{mishura2020existence}
Yuliya Mishura and Alexander Veretennikov.
\newblock Existence and uniqueness theorems for solutions of mckean--vlasov
  stochastic equations.
\newblock {\em Theory of Probability and Mathematical Statistics}, 2020.

\bibitem[Ond04]{ondrejat2004uniqueness}
Martin Ondrej\'{a}t.
\newblock Uniqueness for stochastic evolution equations in {B}anach spaces.
\newblock {\em Dissertationes Math. (Rozprawy Mat.)}, 426:63, 2004.

\bibitem[PR07]{prevot2007spde}
Claudia Pr\'{e}v\^{o}t and Michael R\"{o}ckner.
\newblock {\em {A} {C}oncise {C}ourse on {S}tochastic {P}artial {D}ifferential
  {E}quations}, volume 1905 of {\em Lecture Notes in Mathematics}.
\newblock Springer, Berlin, 2007.

\bibitem[RXZ20]{roeckner2020levy}
Michael R\"{o}ckner, Longjie Xie, and Xicheng Zhang.
\newblock Superposition principle for non-local {F}okker-{P}lanck-{K}olmogorov
  operators.
\newblock {\em Probab. Theory Related Fields}, 178(3-4):699--733, 2020.

\bibitem[RZ10]{roeckner2010weakuniqueness}
Michael R\"{o}ckner and Xicheng Zhang.
\newblock Weak uniqueness of {F}okker-{P}lanck equations with degenerate and
  bounded coefficients.
\newblock {\em C. R. Math. Acad. Sci. Paris}, 348(7-8):435--438, 2010.

\bibitem[RZ21]{zhang2021ddsde}
Michael Röckner and Xicheng Zhang.
\newblock {Well-posedness of distribution dependent {SDE}s with singular
  drifts}.
\newblock {\em Bernoulli}, 27(2):1131 -- 1158, 2021.

\bibitem[Sch87]{scheutzow1987}
Michael Scheutzow.
\newblock Uniqueness and nonuniqueness of solutions of {V}lasov-{M}c{K}ean
  equations.
\newblock {\em J. Austral. Math. Soc. Ser. A}, 43(2):246--256, 1987.

\bibitem[Szn84]{sznitman1984}
Alain-Sol Sznitman.
\newblock Nonlinear reflecting diffusion process, and the propagation of chaos
  and fluctuations associated.
\newblock {\em J. Funct. Anal.}, 56(3):311--336, 1984.

\bibitem[Tre16]{trevisan_super}
Dario Trevisan.
\newblock Well-posedness of multidimensional diffusion processes with weakly
  differentiable coefficients.
\newblock {\em Electron. J. Probab.}, 21:Paper No. 22, 41, 2016.

\bibitem[Ver81]{Veretennikov_1981}
Alexander Veretennikov.
\newblock {O}n strong solutions and explicit formulas for solutions of
  stochastic integral equations.
\newblock {\em Mathematics of the {USSR}-Sbornik}, 39(3):387--403, April 1981.

\bibitem[YW71]{yamada1971yamada}
Toshio Yamada and Shinzo Watanabe.
\newblock On the uniqueness of solutions of stochastic differential equations.
\newblock {\em J. Math. Kyoto Univ.}, 11:155--167, 1971.

\bibitem[Zha11]{zhang2011singular}
Xicheng Zhang.
\newblock Stochastic homeomorphism flows of {SDE}s with singular drifts and
  {S}obolev diffusion coefficients.
\newblock {\em Electron. J. Probab.}, 16:no. 38, 1096--1116, 2011.

\bibitem[Zie89]{ziemer89}
William~P. Ziemer.
\newblock {\em Weakly differentiable functions : Sobolev spaces and functions
  of bounded variation}.
\newblock Graduate texts in mathematics ; Volume 120. Springer Science+Business
  Media, New York, 1989.

\end{thebibliography}
\end{document}